\documentclass[10pt]{article}
\usepackage{url}
\usepackage{amsmath}
\usepackage{amssymb}
\usepackage{amsthm,enumitem}
\usepackage{empheq}
\usepackage{latexsym}
\usepackage{todonotes}
\usepackage{appendix}
\usepackage{color} 
\usepackage[utf8]{inputenc}
\usepackage[T1]{fontenc}
\DeclareUnicodeCharacter{014D}{\=o}
\usepackage{geometry}
\newtheorem{Theorem}{Theorem}[part]
\newtheorem{Definition}{Definition}[part]
\newtheorem{Proposition}{Proposition}[part]

\newtheorem{Assumption}{Assumption}[part]

\newtheorem{Lemma}{Lemma}[part]

\newtheorem{Remark}{Remark}[part]
\newtheorem{Example}{Example}[part]

\makeatletter \@addtoreset{equation}{section}

\@addtoreset{Definition}{section}

\@addtoreset{Theorem}{section}

\@addtoreset{Proposition}{section}

\@addtoreset{Property}{section}

\@addtoreset{Assumption}{section}

\@addtoreset{Corollary}{section}

\@addtoreset{Lemma}{section}

\@addtoreset{Remark}{section}

\@addtoreset{Example}{section}

\@addtoreset{Condition}{section}

\newcommand{\No}[1]{\left\|#1\right\|}     
\newcommand{\abs}[1]{\left|#1\right|}     

\def \E{\mathbb{E}}
\def \F{\mathbb{F}}

\def \N{\mathbb{N}}
\def \P{\mathbb{P}}

\def \R{\mathbb{R}}

\def\Fc{{\cal F}}

\def\Kc{{\cal K}}

\def\Mc{{\cal M}}

\def\Pc{{\cal P}}

\def\Tc{{\cal T}}

\def\Yc{{\cal Y}}
\def\Zc{{\cal Z}}

\def \Sum{\displaystyle\sum}

\def\esup{{\rm ess \, sup}}

\def\1{{\bf 1}}
\def \ep{\hbox{ }\hfill{ ${\cal t}$~\hspace{-5.1mm}~${\cal u}$   } }

\def \proof{{\noindent \bf Proof. }}
\def \ep{\hbox{ }\hfill$\Box$}

\def\eps{\varepsilon}

\def\reff#1{{\rm(\ref{#1})}}

\def\0{\mathbf{0}}

 \title{Corrigendum for "Second-order reflected backward stochastic differential equations" and "Second-order BSDEs with general reflection and game options under uncertainty"\footnote{The authors would like to express their gratitude towards Roxana Dumitrescu, Ying Hu, Marie-Claire Quenez, Abdoulaye Soumana Hima, and especially to Jianfeng Zhang for directing us to the gap in our previous work.}}

 \author{Anis {\sc Matoussi} \thanks{CMAP, Ecole Polytechnique, Paris, and Universit\'e du Maine, Le Mans, \texttt{anis.matoussi@univ-lemans.fr}.} \and Dylan {\sc Possama\"{i}} \footnote{Universit\'e Paris--Dauphine, PSL Research University, CNRS, CEREMADE, 75016 Paris, France, \texttt{possamai@cere-} \texttt{made.dauphine.fr}.} \and Chao {\sc Zhou} \footnote{Department of Mathematics National University of Singapore, \texttt{matzc@nus.edu.sg}}}
 
             \date{\today}

\begin{document}

\maketitle

\begin{abstract}
\noindent The aim of this short note is to fill in a gap in our earlier paper \cite{matoussi2013second} on {\rm {\rm 2BSDE}s} with reflections, and to explain how to correct the subsequent results in the second paper \cite{matoussi2014second}. We also provide more insight on the properties of 2{\rm RBSDE}s, in the light of the recent contributions \cite{li2017reflected,soumana2017equations} in the so-called $G-$framework. \vspace{5mm}

\vspace{5mm}


\end{abstract}

\section{Introduction}
In this short note, we fill in a gap in our earlier wellposedness result on so-called second-order reflected {\rm {\rm 2BSDE}s} (2{\rm RBSDE}s for short). The issue stemmed from a wrongly defined minimality condition which ensures uniqueness of the solution, which we correct here. We also use this occasion to prove that an alternative minimality condition, taking the form of a Skorokhod-like condition, leads to the exact same solution, provided that an additional assumption linked to the oscillations of the lower obstacle is added (see Assumption \ref{assum} below). This new condition appeared recently in the two contributions \cite{li2017reflected,soumana2017equations} on reflected $G-$BSDEs, and our result proves that the two notions do coincide, and that our formulation produces more general results. 

\vspace{0.5em}
Since some of the results of \cite{matoussi2013second} were used in our subsequent paper \cite{matoussi2014second} considering doubly reflected {\rm {\rm 2BSDE}s}, we also explain how to change the minimality condition there, as well as which results are impacted by this change. Roughly speaking, all our previous results still hold true, except for the {\it a priori} estimates, where we no longer control the total variation of the bounded variation process appearing in the solution, but only its $\mathbb D^{2,\kappa}_H-$norm.

\section{The lower reflected case}
The notations in this section are the ones in \cite{matoussi2013second}. Our only change is that to remain coherent with the notations in the next section, we will denote the lower obstacle in the 2{\rm RBSDE} by $L$ instead of $S$.

\subsection{The gap}
The mistake in the paper \cite{matoussi2013second} can be found right after Equation (3.3), when we try to prove that the process $K^{\P'}-k^{\P'}$ is non-decreasing. Indeed, appealing to the minimality condition $(2.6)$ in \cite{matoussi2013second} does not imply the sub--martingality of this process, since it only gives the required inequality for any $t\in[0,T]$ and the {\it fixed} time $T$. The end of Step $(ii)$ of the proof of Theorem 3.1 in \cite{matoussi2013second} therefore does not go through. The issue here is that the minimality condition $(2.6)$ that we wrote is only adapted to the case where the generator $F$ is $0$, described in Remark 3.1 in \cite{matoussi2013second}, and it has to be modified. One could argue there that $K^{\P'}-k^{\P'}$ might still be non-decreasing, and that an appropriate minimality condition should reflect this fact. 

\vspace{0.5em}
However, $K^{\P'}-k^{\P'}$ is not non--decreasing in general. The issue was actually partially pointed out in Remark 3.6 of \cite{matoussi2014second}. It is explained there (and the proof of this result is independent of the mistake in the minimality condition) that on the event $\{Y_{t^-}=L_{t^-}\}$, one has $K^\P=k^\P$, $\P-a.s.$, for any $\P\in\Pc_H^\kappa$, meaning that $K^\P-k^\P$ is constant (and thus non--decreasing) as long as $Y_{t^-}=L_{t^-}$. Similarly, the Skorokhod condition satisfied by $k^\P$ implies that $K^\P-k^\P$ is still non-decreasing on the event $\{y^\P_{t^-}>L_{t^-}\}$. But there is nothing we can say on the event $\{Y_t>y^\P_{t^-}=L_{t^-}\}$. Also, the following counter--example, communicated to us by Jianfeng Zhang, proves that $K^{\P'}-k^{\P'}$ is not non--decreasing in general.
 
 \begin{Example}[Jianfeng Zhang]
 Fix $T=2$ and take as a lower obstacle a process $L$ satisfying the required assumptions in {\rm \cite{matoussi2013second}} as well as
 $$L_t:=2(1-t),\; 0\leq t\leq 1,\; \text{and}\; L_t\leq 2,\; 1\leq t\leq 2.$$
 Furthermore, take the generator $F$ of the {\rm $2$RBSDE} to be $0$, and the terminal condition to be $L_2$. In this case, the solution to the 2{\rm RBSDE} being necessarily the supremum of the solutions to the associated {\rm RBSDE}s, we will have automatically the representations
 $$Y_t=\underset{\mathbb P'\in\Pc^\kappa_H(t^+,\P)}{{\rm essup}^\P}\; \underset{\tau\in\Tc_{t,T}}{{\rm essup}^\P}\; \mathbb E^{\mathbb P'}[L_\tau|\Fc_T],\; y_t^\P=\underset{\tau\in\Tc_{t,T}}{{\rm essup}^\P}\;\mathbb E^{\mathbb P'}[L_\tau|\Fc_T].$$
 Furthermore, in this case since $F=0$, $K^{\P'}-k^{\P'}$ being a $\P'-$sub--martingale is equivalent to $Y-y^{\P'}$ being a $\P'$-supermartingale, which would imply in particular that
 \begin{equation}\label{eq:counter}
 Y_0-y_0^{\P'}\geq \mathbb E^{\mathbb P'}\big[Y_1-y_1^{\P'}\big].
 \end{equation}
 However, it is clear by definition of $L$ that $Y_0=y_0^{\P'}=2$. However, there is absolutely no reason why in general one could not have, for some $\P'$, and for an appropriate choice of $L$, $Y_1>y_1^{\P'}$ $($recall that we always have $Y_1\geq y_1^{\P'})$, at least with strictly positive $\P'-$probability, which then contradicts \eqref{eq:counter}. 
 \end{Example}
 
\subsection{The new minimality condition and uniqueness}
 This being clarified, let us now explain what should be the appropriate minimality condition replacing (2.6) in \cite{matoussi2013second}. Using the Lipschitz property of $F$ (see Assumption 2.3(iii) in \cite{matoussi2013second}), we can define bounded functions $\lambda:[0,T]\times\Omega\times\R\times\R\times\R^d\times\R^d\times D_H\longrightarrow R$ and $\eta:[0,T]\times\Omega\times\R\times\R\times\R^d\times\R^d\times D_H\longrightarrow R^d$ such that for any $(t,\omega,y,y',z,z',a)$
   \begin{equation}\label{eqF}
 F_t(\omega,y,z,a)-F_t(\omega,y',z',a)=\lambda_t(\omega,y,y',z,z',a)(y-y')+\eta_t(\omega,y,y',z,z',a)\cdot a^{1/2}(z-z').
  \end{equation}
 Define then for any $\P\in\Pc^\kappa_H$, and for any $t\in[0,T]$ the process
  \begin{equation}\label{eqM}
 M_s^{t,\P}:=\exp\left(\int_t^s\bigg(\lambda_u-\frac12|\eta_u|^2\bigg)(Y_u,y_u^\P,Z_u,z_u^\P,\widehat a_u)du-\int_t^s\eta_u(Y_u,y_u^\P,Z_u,z_u^\P,\widehat a_u)\cdot \widehat a_u^{-1/2}dB_u\right).
 \end{equation}
 Following the arguments in the beginning of the proof of Theorem 3.1 in \cite{matoussi2013second}, we have then for any $t\in[0,T]$, any $\P\in\Pc_H^\kappa$ and any $\P'\in\Pc_H^\kappa(t^+,\P)$
 \begin{equation}\label{eq}
 Y_t-y_t^{\P'}=\mathbb E_t^{\P'}\left[\int_t^TM^{t,\P'}_sd\big(K^{\P'}_s-k^{\P'}_s\big)\right],\; \P-a.s.
 \end{equation}
 Therefore, the representation formula 
 $$Y_t=\underset{\mathbb P'\in\Pc^\kappa_H(t^+,\P)}{{\rm essup}^\P}\; y_t^{\P'},\; \P-a.s.,$$
 is equivalent to the new minimality condition
\begin{equation}\label{eq:new}
\underset{\mathbb P'\in\Pc^\kappa_H(t^+,\P)}{{\rm essinf}^\P}\; \mathbb E_t^{\P'}\left[\int_t^TM^{t,\P'}_sd\big(K^{\P'}_s-k^{\P'}_s\big)\right]=0,\; \P-a.s.
\end{equation}
If one replaces the minimality condition (2.6) in \cite{matoussi2013second} by \eqref{eq:new} above, as well as in the statement of Theorem 3.1 in \cite{matoussi2013second} the representation $(3.1)$ by simply
\begin{equation}\label{eq:rep}
Y_t=\underset{\mathbb P'\in\Pc^\kappa_H(t^+,\P)}{{\rm essup}^\P}\; y_t^{\P'},\; \P-a.s.,
\end{equation}
then the proof of \eqref{eq:rep} is immediate as soon as one has proved \eqref{eq}. This allows us to recover uniqueness of the solution (see Section \ref{sec.detailed} below for details).

\begin{Remark}
The representation formula $(3.1)$ in {\rm \cite{matoussi2013second}} does not only involve $t$ and $T$, but any pair $0\leq t\leq s\leq T$. If one only assumes the new minimality condition \eqref{eq:new}, then it cannot be proved immediately that 
$$Y_t=\underset{\mathbb P'\in\Pc^\kappa_H(t^+,\P)}{{\rm essup}^\P}\; 
y_t^{\P'}(s,Y_s),\; \P-a.s.$$
However, once we have proved that the solution $Y$ of the {\rm 2RBSDE} satisfies the dynamic programming principle, then the above is immediate. Furthermore, the case $s=T$ is enough to obtain uniqueness, which is the purpose of Theorem 3.1 in {\rm \cite{matoussi2013second}}.
\end{Remark}

\begin{Remark}
Since in general $K^{\P'}-k^{\P'}$ is not non--decreasing, we cannot reduce \eqref{eq:new} to a statement involving only $K^{\P'}$ and $k^{\P'}$, as is the case for non--reflected {\rm {\rm 2BSDE}s}, see for instance {\rm \cite{soner2012wellposedness}}. However, when $L=-\infty$ and there is no reflection, $k^{\P'}$ becomes identically $0$, and \eqref{eq:new} is indeed equivalent to 
$$\underset{\mathbb P'\in\Pc^\kappa_H(t^+,\P)}{{\rm essinf}^\P}\; \mathbb E_t^{\P'}\left[K^{\P'}_T-K^{\P'}_t\right]=0,\; \P-a.s.,$$
see the arguments in Step $(ii)$ of the proof of {\rm \cite[Theorem 4.3, Theorem 4.6]{soner2012wellposedness}}.

\vspace{0.5em}
The need to depart from the "standard" minimality condition has also been pointed out by Popier and Zhou {\rm \cite{popier2019second}}, when dealing with {\rm {\rm 2BSDE}s} under a monotonicity condition, with hypotheses relaxing the earlier work {\rm\cite{possamai2013second1}}. 
\end{Remark}

\begin{Remark}
As a sanity check, let us verify here that the new minimality condition \eqref{eq:new} indeed allows to recover the classical {\rm RBSDE} theory when $\Pc^\kappa_H$ is reduced to a singleton $\{\P\}$. In this case, \eqref{eq:new} says exactly that the bounded variation process $\int_0^\cdot M^{0,\P}_sd\big(K^{\P}_s-k^{\P}_s\big)$ is a $\P$--martingale. Since the filtration $\overline{\F}^\P$ satisfies the predictable martingale representation property, it means that this process is identically $0$. Now since $M^{0,\P}$ is $\P-a.s.$ positive, this implies that $K^\P=k^\P$, which is the desired property. 
\end{Remark}

\subsection{Recovering existence}
The second instance of the use of the wrong conclusion that $K^\P-k^\P$ was non-decreasing in \cite{matoussi2013second} is in the existence proof, during the discussion after Equation $(4.6)$. At this point, the last thing to prove is that $K$ satisfies the new minimality condition \eqref{eq:new}. However, we already have the result of Proposition 4.2 in \cite{matoussi2013second} which shows that the process $V^+$ satisfies the representation formula \eqref{eq:rep}. Therefore, the fact that \eqref{eq:new} is indeed satisfied is immediate, since both statements are equivalent (see Section \ref{sec.detailed} below for details).

\vspace{0.5em}
To summarise, one should replace Definition 2.3 in \cite{matoussi2013second} by the following, and use the corrections explained above in the proofs.

\begin{Definition}\label{def1}
For $\xi \in \mathbb L^{2,\kappa}_H$, we say $(Y,Z)\in \mathbb D^{2,\kappa}_H\times\mathbb H^{2,\kappa}_H$ is a solution to the {\rm 2RBSDE} if
\begin{itemize}
\item[$\bullet$] $Y_T=\xi$, and $Y_t\geq L_t$, $t\in[0,T]$, $\mathcal P_H^\kappa-q.s$.
\item[$\bullet$] $\forall \mathbb P \in \mathcal P_H^\kappa$, the process $K^{\mathbb P}$ defined below has non-decreasing paths $\mathbb P-a.s.$
\begin{equation}
K_t^{\mathbb P}:=Y_0-Y_t - \int_0^t\widehat{F}_s(Y_s,Z_s)ds+\int_0^tZ_sdB_s, \text{ } 0\leq t\leq T, \text{  } \mathbb P-a.s.
\label{2BSDE.kref}
\end{equation}

\item[$\bullet$] We have the following minimality condition
\begin{equation*}
\underset{\mathbb P'\in\Pc^\kappa_H(t^+,\P)}{{\rm essinf}^\P}\; \mathbb E_t^{\P'}\left[\int_t^TM^{t,\P'}_sd\big(K^{\P'}_s-k^{\P'}_s\big)\right]=0
, \text{  } \mathbb P-a.s., \text{ } 0\leq t\leq T, \text{ } \forall \mathbb P \in \mathcal P_H^\kappa.
\end{equation*}
\end{itemize}
\end{Definition}

\subsection{Detailed proofs}\label{sec.detailed}
For the ease of the reader, we give the details of the proof for the uniqueness, which is a correction of Theorem 3.1 in {\rm \cite{matoussi2013second}}.

\begin{Theorem}\label{uniqueref}
Let Assumptions 2.1 and 2.2 in {\rm \cite{matoussi2013second}} hold. Assume $\xi \in \mathbb{L}^{2,\kappa}_H$ and that $(Y,Z)$ is a solution to $2${\rm RBSDE} in Definition \ref{def1}. Then, for any $\mathbb{P}\in\mathcal{P}^\kappa_H$ and $0\leq t \leq T$,

\begin{align}
\label{representationref}
Y_{t}&=\underset{\mathbb{P}^{'}\in\mathcal{P}^\kappa_H(t^+,\mathbb{P})}{\esup^\mathbb{P}}y_{t}^{\mathbb{P}^{'}}, \text{ }\mathbb{P}-a.s.
\end{align}

Consequently, the $2${\rm RBSDE} in Definition \ref{def1} has at most one solution in $ \mathbb D^{2,\kappa}_H\times\mathbb H^{2,\kappa}_H$.
\end{Theorem}

\vspace{1em}
\proof
\vspace{0.5em}
We start by proving \reff{representationref}.

\begin{itemize}[leftmargin=*]
\item[\rm{(i)}] Fix $0\leq t\leq T$ and $\mathbb P\in\mathcal P^\kappa_H$. For any $\mathbb P^{'}\in\mathcal P^\kappa_H(t^+,\mathbb P)$, we have
$$Y_{t} = \xi + \int_{t}^{T} \widehat{F}_s(Y_s,Z_s)ds - \int_{t}^{T} Z_sdB_s + K_{T}^{\mathbb P^{'}}-K_{t}^{\mathbb P^{'}},  \text{ } \mathbb P^{'}-a.s.$$

Now, it is clear that we can always decompose the non-decreasing process $K^\mathbb P$ into
$$K^{\mathbb P^{'}}_t=A_t^{\mathbb P^{'}}+B_t^{\mathbb P^{'}},\text{ }\mathbb P^{'}-a.s.,$$
were $A^{\mathbb P^{'}}$ and $B^{\mathbb P^{'}}$ are two non-decreasing processes such that $A^{\mathbb P^{'}}$ only increases when $Y_{t^-}=S_{t^-}$ and $B^{\mathbb P^{'}}$ only increases when $Y_{t^{-}}>S_{t^-}$. With that decomposition, we can apply a generalisation of the usual comparison theorem proved by El Karoui et al. (see Theorem $8.3$ in \cite{el2008backward}), whose proof is given in the appendix of {\rm \cite{matoussi2013second}}, under $\mathbb P^{'}$ to obtain $Y_{t}\geq y_{t}^{\mathbb P^{'}}$ and $A_{T}^{\mathbb P^{'}}-A_{t}^{\mathbb P^{'}}\leq k^{\P^{'}}_{T}-k^{\P^{'}}_{t},\ \mathbb P^{'}-a.s.$ Since $\mathbb P^{'}=\mathbb P$ on $\mathcal F_t^+$, we get $Y_{t}\geq y_{t}^{\mathbb P^{'}}$, $\mathbb P-a.s.$ and thus
$$Y_{t}\geq\underset{\mathbb{P}^{'}\in\mathcal{P}^\kappa_H(t^+,\mathbb{P})}{\esup^\mathbb{P}}y_{t}^{\mathbb{P}^{'}}, \text{ }\mathbb{P}-a.s.$$

\item[\rm{(ii)}] We now prove the reverse inequality. Fix $\mathbb P\in\mathcal P^\kappa_H$. For every $\mathbb P^{'}\in \mathcal P^\kappa_H(t^+,\mathbb P)$, denote
$$\delta Y:=Y-y^{\mathbb P^{'}},\text{ }\delta Z:=Z-z^{\mathbb P^{'}} \text{ and } \delta K^{\mathbb P^{'}}:=K^{\mathbb P^{'}}-k^{\mathbb P^{'}}.$$

As in (\ref{eqF}), there exist two bounded processes $\lambda$ and $\eta$ such that for all $t\leq T$
$$\delta Y_t=\int_t^{T}\left(\lambda_s\delta Y_s+\eta_s\widehat{a}_s^{1/2}\delta Z_s\right)ds-\int_t^{T}\delta Z_sdB_s+\delta K_{T}^{\mathbb P^{'}}-\delta K_{t}^{\mathbb P^{'}},\text{ }\mathbb P^{'}-a.s.$$

\vspace{0.5em}
Then, by It\^o's formula, we obtain
\begin{equation}
\label{brubru}
\delta Y_{t}=\mathbb E_{t}^{\mathbb P^{'}} \left[\int_{t}^{T}M^{t,\mathbb P^{'}}_s d\delta K_s^{\mathbb P^{'}}\right],
\end{equation}
where $M^{t,\mathbb P^{'}}_s$ is defined in (\ref{eqM}). Taking the essential infimum on both sides, then (\ref{eq:new}) implies (\ref{representationref}).
\end{itemize}\ep

For the existence, there are only changes in the case when $\xi$ is in $\rm{UC_b}(\Omega)$. And we only give details for the part where the minimality condition is concerned.  In Proposition 4.2 in {\rm \cite{matoussi2013second}}, we proved that the process $V^+$ satisfies the representation formula 
\begin{equation}\label{eq:repV}
V_t^+=\underset{\mathbb P^{'}\in\mathcal P^{\kappa}_H(t^+,\mathbb P)}{\esup^\mathbb P}y_t^{\mathbb P^{'}}, \text{ }\mathbb P-a.s., \text{ }\forall \mathbb P\in\mathcal P_H^{\kappa}.
\end{equation}

We have to check that the minimality condition \reff{eq:new} holds. Fix $\mathbb P$ in $\mathcal P^\kappa_H$ and $\mathbb P^{'}\in\mathcal P^\kappa_H(t^+,\mathbb P)$. By the Lipschitz property of $F$, we know that there exists bounded processes $\lambda$ and $\eta$ as in (\ref{eqF}) such that
\begin{align}\label{bidulou2}
\nonumber V^{+}_t-y_t^{\mathbb P^{'}}&=\int_t^T\left[\lambda_s(V^{+}_s-y_s^{\mathbb P^{'}})+\eta_s\widehat a_s^{1/2}(\overline Z_s-z_s^{\mathbb P^{'}})\right]ds-\int_t^T\widehat a_s^{1/2}(\overline Z_s-z_s^{\mathbb P^{'}})\widehat a^{-1/2}_sdB_s\\
&+K_T^{\mathbb P^{'}}-K_t^{\mathbb P^{'}}-k_T^{\mathbb P^{'}}+k_t^{\mathbb P^{'}}.
\end{align}

\vspace{0.5em}
Then, by It\^o's formula, we obtain
$$V^+_t-y_t^{\mathbb P^{'}}=\mathbb E^{\mathbb P^{'}}_t\left[\int_t^TM_{s}^{t,\mathbb P^{'}}d(K_s^{\mathbb P^{'}}-k^{\mathbb P^{'}}_s)\right].$$

\vspace{0.5em}
Taking the essential infimum on both sides, \reff{eq:repV} implies that the minimality condition \reff{eq:new} is satisfied.

\medskip
Now since 
$$Y_{t}=\underset{\mathbb{P}^{'}\in\mathcal{P}^\kappa_H(t^+,\mathbb{P})}{\esup^\mathbb{P}}y_{t}^{\mathbb{P}^{'}}, \text{ }\mathbb{P}-a.s., \text{ } t\in [0,T], \text{ for all }\mathbb P\in \mathcal{P}^\kappa_H,$$
$Y$ is thus unique. Then, since we have that $d\left<Y,B\right>_t=Z_td\left<B\right>_t, \text{ } \mathcal{P}^\kappa_H-q.s.$, $Z$ is also unique. Finally, the process $K^\mathbb P$ is uniquely determined. 

\subsection{An alternative: Skorokhod minimality condition}
Readers familiar with the theory of standard reflected BSDEs should be wondering whether there is an equivalent, in the second--order setting, of the so--called Skorokhod condition. The latter states that the non-decreasing process appearing in the definition of a {\rm RBSDE} acts in a minimal way, only when the solution actually reaches the obstacle, and implies uniqueness of the solution (see the seminal paper \cite{el1997reflected} for more details). There are actually two recent papers which treat the very related problem of reflected $G-$BSDEs, namely \cite{li2017reflected,soumana2017equations}, and which use a generalisation of this condition. The aim of this section is to show that this condition also implies wellposedness in our framework, under an additional assumption on the obstacle $L$, and that the two definitions are actually equivalent. We also provide a more detailed comparison between \cite{li2017reflected,soumana2017equations} and our work at the end of the section.

\subsubsection{Wellposedness under Skorokhod condition}
Using the same notations as before, the Skorokhod condition for {\rm 2RBSDE}s reads
\begin{equation}\label{eq:sko}
\underset{\mathbb P'\in\Pc^\kappa_H(t^+,\P)}{{\rm essinf}^\P}\; \mathbb E_t^{\P'}\bigg[\int_t^T\big(Y_{s^-}-L_{s^-}\big)dK_s^{\P'}\bigg]=0,\; t\in[0,T],\; \P-a.s.,\; \forall\P\in\Pc^\kappa_H.
\end{equation}

For ease of reference, we provide the corresponding alternative definition of a solution to the 2{\rm RBSDE}.
\begin{Definition}\label{def2}
For $\xi \in \mathbb L^{2,\kappa}_H$, we say $(Y,Z)\in \mathbb D^{2,\kappa}_H\times\mathbb H^{2,\kappa}_H$ is a Skorokhod--solution to the {\rm 2RBSDE} if

\begin{itemize}
\item[$\bullet$] $Y_T=\xi$, and $Y_t\geq L_t$, $t\in[0,T]$, $\mathcal P_H^\kappa-q.s$.
\item[$\bullet$] $\forall \mathbb P \in \mathcal P_H^\kappa$, the process $K^{\mathbb P}$ defined below has non--decreasing paths $\mathbb P-a.s.$
\begin{equation}
K_t^{\mathbb P}:=Y_0-Y_t - \int_0^t\widehat{F}_s(Y_s,Z_s)ds+\int_0^tZ_sdB_s, \text{ } 0\leq t\leq T, \text{  } \mathbb P-a.s.
\label{2BSDE.kref2}
\end{equation}

\item[$\bullet$] The minimality condition \eqref{eq:sko} holds
\end{itemize}
\end{Definition}

In more mundane terms, this condition is saying that if there is a probability measure $\P$ such that the supremum in the representation formula \eqref{eq:rep} is attained, then on the support of $\P$, the classical Skorokhod condition is satisfied by the solution of the 2{\rm RBSDE}. 

\vspace{0.5em}
Let us now argue how \eqref{eq:sko} can be used instead of \eqref{eq:new} to recover wellposedness, and that both conditions actually lead to the exact same solution. Notice however that the method of proof here requires the following  condition on $L$, which basically asks that the variations of $L$ are not too "extreme".
\begin{Assumption}\label{assum}
We have for any $m\in\mathbb N$, for any sequence $\{(t_i^n)_{1\leq i\leq n},\; n\geq 0\}$ of partitions of $[0,T]$ $($allowing for stopping times$)$ whose mesh goes to $0$ as $n$ goes to $+\infty$, and for any $\eps>0$
$$\underset{n\rightarrow+\infty}{\lim}\; \underset{\P\in\Pc^\kappa_H}{\sup}\P\bigg[\sum_{i=0}^{n-1}{\bf 1}_{\big\{|L_{t_{i+1}^{n-}}-L_{t_i^{n-}}|\geq\eps\big\}}\geq n-m\bigg]=0.$$
\end{Assumption}
The above assumption puts somehow restrictions on the oscillations or the variations of $L$. Before pursuing, let us give the following two sufficient conditions.

\begin{Lemma}\label{lemma}
Either of the following two conditions imply Assumption \ref{assum}

\vspace{0.5em}
$(i)$ For any non--decreasing sequence of stopping times $(\rho_n)_{n\geq 0}$ converging to $T$ and for any $\eps>0$, we have
$$ \underset{n\rightarrow+\infty}{\lim}\; \underset{\P\in\Pc^\kappa_H}{\sup}\P\big[\big|L_{\rho_{n+1}^-}-L_{\rho_n^-}\big|\geq \eps\big]=0.$$ 

$(ii)$ Let $\Pi_{[0,T]}$ be the set of all partitions of $[0,T]$ $($which allow for stopping times$)$. We have for some $p\geq 1$
$$\ell:=\underset{(\rho_i)_i\in\Pi_{[0,T]}}{\sup}\; \underset{\P\in\Pc_H^\kappa}{\sup}\mathbb E^\P\bigg[\sum_{i=0}^{n-1}\abs{L_{\rho_{i+1}^-}-L_{\rho_i^-}}^p\bigg]<+\infty.$$
\end{Lemma}

\begin{proof}
For $(i)$, It actually suffices to notice that for any $\P\in\Pc_H^\kappa$, for any $m\leq n$ and for any $\eps>0$
$$\P\bigg[\sum_{i=0}^{n-1}{\bf 1}_{\big\{|L_{t_{i+1}^{n-}}-L_{t_i^{n-}}|\geq\eps\big\}}\geq n-m\bigg]\leq \sum_{i=n-m}^n\P\Big[|L_{t_{i+1}^{n-}}-L_{t_i^{n-}}|\geq\eps\Big].$$
As for $(ii)$, a simple application of Markov inequality provides for any $p\geq 1$
\begin{align*}
\P\bigg[\sum_{i=0}^{n-1}{\bf 1}_{\big\{|L_{t_{i+1}^{n-}}-L_{t_i^{n-}}|\geq\eps\big\}}\geq n-m\bigg]&\leq \P\bigg[\sum_{i=0}^{n-1}|L_{t_{i+1}^{n-}}-L_{t_i^{n-}}|^p\geq (n-m)\eps^p\bigg]\\
&\leq \frac{1}{\eps^p(n-m)}\E^\P\left[\sum_{i=0}^{n-1}|L_{t_{i+1}^{n-}}-L_{t_i^{n-}}|^p\right]\leq \frac{\ell}{\eps^p(n-m)}.
\end{align*}
\end{proof}

\begin{Remark}
Obviously,  conditions in Lemma \ref{lemma}, and thus Assumption  \ref{assum} are satisfied  for $L$  being  a semi--martingale of the form
$$L_t=L_0+\int_0^tU_sds+\int_0^tV_sdB_s+C_t,\text{ }\mathcal P_H^\kappa-q.s.$$
where $C$ is c\`adl\`ag process of integrable variation such that the measure $dC_t$ is singular with respect to the Lebesgue measure $dt$ and which admits the following decomposition
$$C_t=C_t^+-C_t^-,$$
where $C^+$ and $C^-$ are non--decreasing processes. Besides, $U$ and $V$ are respectively $\mathbb R-$ and $\mathbb R^d-$valued, and $\mathbb F-$progressively measurable processes such that for some  $ p \geq 1$, 
$$  \underset{\P\in\Pc^\kappa_H}{\sup} \mathbb E^\P\bigg[\sup_{0 \leq  t \leq T}  \abs{U_t}^p + \Big(\int_0^T \abs{V_t}^2dt \Big)^{p/2}  + (C_T^+)^p +(C_T^-)^p \,   \bigg]<  +\infty.$$

\end{Remark}

\begin{Remark}
Notice that in the proof of Theorem \ref{th:main}, we actually only use Assumption \ref{assum} for a specific partition of $[0,T]$ defined in terms of successive crossings of $Y-L$. We therefore could have formulated Assumption \ref{assum} in terms of this partition only. We choose not to do so in order to make the assumption more natural.
\end{Remark}
The main result is now as follows, and its proof borrows a lot from the seminal paper of Ekren, Touzi and Zhang \cite{ekren2014optimal}.

\begin{Theorem}\label{th:main}
Let Assumption \ref{assum} hold, as well as the necessary assumptions for wellposedness in {\rm \cite{matoussi2013second}}. Then there is a unique Skorokhod--solution to the {\rm 2RBSDE} which coincides with the unique solution to the {\rm 2RBSDE}.
\end{Theorem}

\begin{proof}
We now argue in two steps.

\vspace{0.5em}
\hspace{2em}{\bf Step $1$: uniqueness}

\vspace{0.5em}
This is the easiest part. Assume that there exists a Skorokhod--solution $(\widetilde Y,\widetilde Z)$. We will argue that $\widetilde Y=Y$, which implies immediately that $\widetilde Z=Z$, since $Z$ is uniquely defined by the quadratic co--variation between $Y$ and $B$.

\vspace{0.5em}
Fix first some $\P\in\Pc_H^\kappa$. By definition, $\widetilde Y$ is a super--solution under $\P$ to the standard BSDE with terminal condition $\xi$, generator $\widehat F$. Since it is also always above $L$, and since solutions to reflected BSDEs are also the minimal super--solutions of the associated BSDEs, we deduce that necessarily we have $\widetilde Y\geq y^\P,\; \P-a.s.$, which implies by arbitrariness of $\P$ and by \eqref{eq:rep} that
$$\widetilde Y\geq Y.$$
For the converse inequality, fix some $\eps>0$ and some $\P\in\Pc_H^\kappa$, and define the following stopping time
$$\tilde\tau_\eps
:=\inf\{t\geq 0,\; \widetilde Y_{t^-}-L_{t^-}\leq \eps\}\wedge T.$$
The Skorokhod condition implies that
$$0=\underset{\mathbb P'\in\Pc^\kappa_H(t^+,\P)}{{\rm essinf}^\P}\; \mathbb E^{\P'}_t\bigg[\int_t^T\big(\widetilde Y_{s^-}-L_{s^-}\big)dK_s^{\P'}\bigg]\geq \eps \underset{\mathbb P'\in\Pc^\kappa_H(t^+,\P)}{{\rm essinf}^\P}\; \mathbb E^{\P'}_t\big[K_{\tilde\tau_\eps}^{\P'}-K_t^{\P'}\big].$$
Next, let $(\mathcal Y^{\P'}(\tilde\tau_\eps, L_{\tilde\tau_\eps}),\mathcal Z^{\P'}(\tilde\tau_\eps, L_{\tilde\tau_\eps}))$ be the solution, under $\P'$ and on $[0,\tilde\tau_\eps]$, of the BSDE with generator $\widehat F$ and terminal condition $L_{\tilde\tau\eps}$. We have $\P'-a.s.$
\begin{align*}
\widetilde Y_t&=\widetilde Y_{\tilde\tau_\eps} + \int_t^{\tilde\tau_\eps}\widehat{F}_s\big(\widetilde Y_s,\widetilde Z_s\big)ds-\int_t^{\tilde\tau_\eps}Z_sdB_s+\int_t^{\tilde\tau\eps}dK^{\P'}_s,\; 0\leq t\leq \tilde\tau_\eps,\\
{\color{black}\Yc^{\P'}_t(\tilde\tau_\eps, L_{\tilde\tau_\eps})}&=L_{\tilde\tau_\eps}\mathbf{1}_{\{\tilde\tau\eps<T\}}+\xi\mathbf{1}_{\{\tilde\tau_\eps=T\}} + \int_t^{\tilde\tau_\eps}\widehat{F}_s\big(\Yc^{\P'}_s(\tilde\tau_\eps, L_{\tilde\tau_\eps}),\Zc^{\P'}_s(\tilde\tau_\eps, L_{\tilde\tau_\eps})\big)ds-\int_t^{\tilde\tau_\eps}\Zc^{\P'}_s(\tilde\tau_\eps, L_{\tilde\tau_\eps})dB_s,\; 0\leq t\leq \tilde\tau_\eps.
\end{align*}
We can now use classical linearisation arguments as above to show the result. First, we define $M^{t,\mathbb P^{'}}_s$ in the same way as in (\ref{eqM}). Notice that as in \cite[Equation (4.12)]{soner2012wellposedness}, for any $p\geq 1$, there exists $C_p>0$ such that
$$\mathbb E^{\P^{'}}_t\left[\underset{t\leq s\leq T}{\sup} (M^{t,\mathbb P^{'}}_{s})^p\right]\leq C_p, \; \P^{'}-a.s.$$
Then, by It\=o's formula, we obtain
\begin{align*}
\widetilde Y_t-\Yc^{\P'}_t(\tilde\tau_\eps, L_{\tilde\tau_\eps})&= \mathbb E^{\P'}_t\bigg[M^{t,\mathbb P^{'}}_{\tilde\tau_\eps}\Big(\widetilde Y_{\tilde\tau_\eps}- \Yc^{\P'}_{\tilde\tau_\eps}(\tilde\tau_\eps, L_{\tilde\tau_\eps})\Big) +\int_{t}^{{\tilde\tau_\eps}}M^{t,\mathbb P^{'}}_s dK^{\P'}_s\bigg]\\
&\leq \mathbb E^{\P'}_t\bigg[\underset{t\leq s\leq T}{\sup} M^{t,\mathbb P^{'}}_{s}\left(\widetilde Y_{\tilde\tau_\eps}- \Yc^{\P'}_{\tilde\tau_\eps}(\tilde\tau_\eps, L_{\tilde\tau_\eps}) + K^{\P'}_{\tilde\tau_\eps}-K^{\P'}_t\right)\bigg]\\
&\leq
\eps \mathbb E^{\P'}_t\bigg[\underset{t\leq s\leq T}{\sup} M^{t,\mathbb P^{'}}_{s}\bigg] +\mathbb E^{\P'}_t\bigg[\underset{t\leq s\leq T}{\sup} M^{t,\mathbb P^{'}}_{s}\left(K^{\P'}_{\tilde\tau_\eps}-K^{\P'}_t\right)\bigg]\\
&\leq
\eps\mathbb E^{\P'}_t\left[\underset{t\leq s\leq T}{\sup} M^{t,\mathbb P^{'}}_{s}\right] +\left(\mathbb E^{\P'}_t\bigg[\underset{t\leq s\leq T}{\sup} \big(M^{t,\mathbb P^{'}}_{s}\big)^3\bigg]\right)^{1/3}
\left(\mathbb E^{\P'}_t \left[\left(K^{\P'}_{\tilde\tau_\eps}-K^{\P'}_t\right)^{3/2}\right]\right)^{2/3}\\
&\leq
\eps \mathbb E^{\P'}_t\left[\underset{t\leq s\leq T}{\sup} M^{t,\mathbb P^{'}}_{s}\right] +\left(\mathbb E^{\P'}_t\left[\underset{t\leq s\leq T}{\sup} \big(M^{t,\mathbb P^{'}}_{s}\big)^3\right]\right)^{1/3}
\left(\mathbb E^{\P'}_t \left[K^{\P'}_{\tilde\tau_\eps}-K^{\P'}_t\right]\mathbb E^{\P'}_t \left[\left(K^{\P'}_{\tilde\tau_\eps}-K^{\P'}_t\right)^{2}\right]\right)^{1/3}\\
&\leq
C_1\eps +(C_3)^{1/3}\left(\mathbb E^{\P'}_t\Big[K^{\P'}_{\tilde\tau_\eps}-K^{\P'}_t\Big]\right)^{1/3}\left(\mathbb E^{\P'}_t\Big[\left(K^{\P'}_{\tilde\tau_\eps}-K^{\P'}_t\right)^2\Big]\right)^{1/3}.
\end{align*}
Define for simplicity
$$C_{t}^{\P}:= \underset{\mathbb P'\in\Pc^\kappa_H(t^+,\P)}{{\rm esssup}^\P}\;\mathbb E^{\P'}_t\Big[\left(K^{\P'}_{\tilde\tau_\eps}-K^{\P'}_t\right)^2\Big].$$ 
Taking the essential infimum on both sides of the previous inequality, we deduce
\begin{align*}
\widetilde Y_t-\underset{\mathbb P'\in\Pc^\kappa_H(t^+,\P)}{{\rm essup}^\P}\Yc^{\P'}_t(\tilde\tau_\eps, L_{\tilde\tau_\eps})&\leq \eps C_1+ (C_3)^{1/3}\big(C_t^\P\big)^{1/3}\bigg( \underset{\mathbb P'\in\Pc^\kappa_H(t^+,\P)}{{\rm essinf}^\P}\;\mathbb E^{\P'}_t \big[K^{\P'}_{\tilde\tau_\eps}-K^{\P'}_t\big]\bigg)^{1/3}.
\end{align*}
Now, as in \cite[Proof of Theorem 4.3 (iii)]{soner2012wellposedness}, since the set of probabilities is upward directed, we obtain that
$$\mathbb E^{\P}\left[\underset{\mathbb P'\in\Pc^\kappa_H(t^+,\P)}{{\rm esssup}^\P}\;\mathbb E^{\P'}_t\Big[\left(K^{\P'}_{\tilde\tau_\eps}-K^{\P'}_t\right)^2\Big]\right]
\leq 
\underset{\mathbb P'\in\Pc^\kappa_H(t^+,\P)}{{\rm sup}}\;\mathbb E^{\P'}\Big[\left(K^{\P'}_{T}-K^{\P'}_t\right)^2\Big]<\infty. 
$$
As such, we can use the Skorokhod condition to deduce now
$$\widetilde Y_t-\underset{\mathbb P'\in\Pc^\kappa_H(t^+,\P)}{{\rm essup}^\P}\Yc^{\P'}_t(\tilde\tau_\eps, L_{\tilde\tau_\eps})\leq \eps C_1.$$
By the classical comparison theorem, we have that $\Yc^{\P'}_t(\tilde\tau_\eps, L_{\tilde\tau_\eps}) \leq {\color{black}y^{\P'}_t}$, so that we deduce
$$\widetilde Y_t\leq \underset{\mathbb P'\in\Pc^\kappa_H(t^+,\P)}{{\rm essup}^\P}\;y^{\P'}_t +C_1\eps\leq Y_t+C_1\eps,$$
which implies the required result by arbitrariness of $\eps$.

\vspace{0.5em}
\hspace{2em}{\bf Step $2$: existence}

\vspace{0.5em}
The only thing that needs to be done here is to prove that the solution we constructed in the sense of Definition \ref{def1} is also a Skorokhod--solution. In other words, we simply have to prove that $Y$ satisfies the Skorokhod minimality condition
\begin{equation*}
\underset{\mathbb P'\in\Pc^\kappa_H(t^+,\P)}{{\rm essinf}^\P}\; \mathbb E_t^{\P'}\bigg[\int_t^T\big(Y_{s^-}-L_{s^-}\big)dK_s^{\P'}\bigg]=0,\; t\in[0,T],\; \P-a.s.,\; \forall\P\in\Pc^\kappa_H.
\end{equation*}

Without loss of generality, we prove that this holds for $t=0$, which is equivalent to proving that
\begin{equation*}
\underset{\mathbb P\in\Pc^\kappa_H}{{\rm inf}}\; \mathbb E^{\P}\bigg[\int_0^T\big(Y_{s^-}-L_{s^-}\big)dK_s^{\P}\bigg]=0.
\end{equation*}
Let us start by fixing some $\eps>0$, and define the following sequence of stopping times $(\tau_n)_{n\geq 1}$ by
$$\tau_0:=0,\; \tau_1:=\inf\{t\geq 0,\; Y_{t^-}-L_{t^-}\leq \eps\}\wedge T,$$
$$ \tau_{2n}:=\inf\{t>\tau_{2n-1},\; Y_{t^-}-L_{t^-}\geq 2\eps\}\wedge T,\; \tau_{2n+1}:=\inf\{t>\tau_{2n},\; Y_{t^-}-L_{t^-}\leq \eps\}\wedge T,\; n\geq 1.$$
We start by proving that for any $n\geq 1$ and any $\P\in\Pc_H^\kappa$
\begin{equation}\label{eq:stop}
\underset{\P\in\Pc_H^\kappa}{\inf}\E^\P\big[K^\P_{\tau_1}\big]=\underset{\P'\in\Pc_H^\kappa(\tau_{2n}^+,\P)}{{\rm essinf}^\P}\E^{\P'}_{\tau_{2n}}\big[K^{\P'}_{\tau_{2n+1}}-K_{\tau_{2n}}^{\P'}\big]=0.
\end{equation}
By the dynamic programming principle (see \cite[Proposition 4.1]{matoussi2013second}), we know that for any $n\geq 0$, and using the link between reflected BSDEs and optimal stopping problems, where for any $0\leq s\leq t\leq T$, $\mathcal T_{s,t}$ denotes the set of stopping times taking values in $[s,t]$
$$Y_{\tau_{2n}}=\underset{\P'\in\Pc_H^\kappa(\tau_{2n}^+,\P)}{{\rm essup}^\P}\; \underset{\tau\in\mathcal T_{\tau_{2n},\tau_{2n+1}}}{{\rm essup}^\P}\Yc^{\P'}_{\tau_{2n}}\big(\tau,L_\tau{\bf 1}_{\{\tau<\tau_{2n+1}\}}+Y_{\tau_{2n+1}}{\bf 1}_{\{\tau=\tau_{2n+1}\}}\big).$$
We claim that the family 
\[
\Big\{\Yc^{\P'}_{\tau_{2n}}\big(\tau,L_\tau{\bf 1}_{\{\tau<\tau_{2n+1}\}}+Y_{\tau_{2n+1}}{\bf 1}_{\{\tau=\tau_{2n+1}\}}\big):(\P',\tau)\in  \Pc_H^\kappa(\tau_{2n}^+,\P)\times \mathcal T_{\tau_{2n},\tau_{2n+1}}\Big\},
\]
is upward directed. Indeed take for $i\in\{1,2\}$, $(\P^{i},\tau^i)\in\Pc_H^\kappa(\tau_{2n}^+,\P)\times \mathcal T_{\tau_{2n},\tau_{2n+1}}$. We need to find a pair $(\P',\tau)\in\Pc_H^\kappa(\tau_{2n}^+,\P)\times \mathcal T_{\tau_{2n},\tau_{2n+1}}$ such that
\begin{align}\label{eq:upward}
\notag \Yc^{\P'}_{\tau_{2n}}\big(\tau,L_\tau{\bf 1}_{\{\tau<\tau_{2n+1}\}}+Y_{\tau_{2n+1}}{\bf 1}_{\{\tau=\tau_{2n+1}\}}\big)&=\max\Big\{\Yc^{\P^1}_{\tau_{2n}}\big(\tau^1,L_{\tau^1}{\bf 1}_{\{\tau^1<\tau_{2n+1}\}}+Y_{\tau_{2n+1}}{\bf 1}_{\{\tau^1=\tau_{2n+1}\}}\big),\\
&\qquad\qquad \Yc^{\P^2}_{\tau_{2n}}\big(\tau^2,L_{\tau^2}{\bf 1}_{\{\tau^2<\tau_{2n+1}\}}+Y_{\tau_{2n+1}}{\bf 1}_{\{\tau^2=\tau_{2n+1}\}}\big)\Big\}.
\end{align}
Define the following sets in $\Fc_{\tau_{2n}^+}^\P$
\[
A:=\Big\{\Yc^{\P^1}_{\tau_{2n}}\big(\tau^1,L_{\tau^1}{\bf 1}_{\{\tau^1<\tau_{2n+1}\}}+Y_{\tau_{2n+1}}{\bf 1}_{\{\tau^1=\tau_{2n+1}\}}\big)\geq \Yc^{\P^2}_{\tau_{2n}}\big(\tau^2,L_{\tau^2}{\bf 1}_{\{\tau^2<\tau_{2n+1}\}}+Y_{\tau_{2n+1}}{\bf 1}_{\{\tau^2=\tau_{2n+1}\}}\big)\Big\},\; B:=\Omega\setminus A.
\]
Define then $\tau:=\tau^1\mathbf{1}_{A}+\tau^2\mathbf{1}_{B}$, as well as $\P^\prime[C]:=\P^1[A\cap C]+\P^2[B\cap C]$, for any $C\in\Fc_T$. It can be checked directly that $\tau\in  \mathcal T_{\tau_{2n},\tau_{2n+1}}$, and that $\P\in\Pc_H^\kappa$ (see similar arguments in \cite[Proof of Theorem 4.2, step $(ii)$]{possamai2018stochastic}), and we have by definition that \eqref{eq:upward} holds for this choice of $(\P^\prime,\tau)$. Therefore, we know that the essential supremum in $Y_{\tau_{2n}}$ is attained along a subsequence, meaning that for any $\delta>0$, there exists some some $\P'_\delta\in\Pc_H^\kappa(\tau_{2n}^+,\P)$ and some $\tau^\delta\in\Tc_{\tau_{2n},\tau_{2n+1}}$ such that
$$Y_{\tau_{2n}}\leq \Yc_{\tau_{2n}}^{\P'_\delta}\big(\tau^\delta,L_{\tau^\delta}{\bf 1}_{\{\tau^\delta<\tau_{2n+1}\}}+Y_{\tau_{2n+1}}{\bf 1}_{\{\tau^\delta=\tau_{2n+1}\}}\big)+\delta.$$
Recall $\eta $ and $\lambda$ from \eqref{eqF}, and define then (notice that this process is slightly different from $M$ defined in \eqref{eqM})
\begin{align*}
\Mc_s^{t,\P'_\delta}:= \exp\bigg(\int_t^s(\lambda_u-\frac12|\eta_u|^2)\big(Y_u,\Yc_u^{\P'_\delta},Z_u,\Zc_u^{\P'_\delta},\widehat a_u)du-\int_t^s\eta_u(Y_u,\Yc_u^{\P'_\delta},Z_u,\Zc_u^{\P'_\delta},\widehat a_u)\cdot \widehat a_u^{-1/2}dB_u\bigg),
\end{align*}
where we denoted for simplicity
$$\Yc^{\P'_\delta}:=\Yc^{\P'_\delta}\big(\tau^\delta,L_{\tau^\delta}{\bf 1}_{\{\tau^\delta<\tau_{2n+1}\}}+Y_{\tau_{2n+1}}{\bf 1}_{\{\tau^\delta=\tau_{2n+1}\}}\big),\; \Zc^{\P'_\delta}:=\Zc^{\P'_\delta}\big(\tau^\delta,L_{\tau^\delta}{\bf 1}_{\{\tau^\delta<\tau_{2n+1}\}}+Y_{\tau_{2n+1}}{\bf 1}_{\{\tau^\delta=\tau_{2n+1}\}}\big).$$
Notice that for any $p\in\mathbb R$, the boundedness of $\lambda$ and $\eta$ imply that for some constant $C_p>0$
$$\underset{\P\in\Pc_H^\kappa}{\sup}\mathbb E^{\P}\bigg[\Big(\underset{t\leq s\leq T}{\sup}\Mc_s^{t,\P}\Big)^p+\Big(\underset{t\leq s\leq T}{\inf}\Mc_s^{t,\P}\Big)^p\bigg]\leq C_p.$$
Then, linearization arguments similar to the ones used before in this note imply that
$$Y_t-\Yc^{\P'_\delta}_t=\mathbb E^{\P'_\delta}_t\bigg[\Mc_{\tau^\delta}^{t,\P'_\delta}(Y_{\tau^\delta}-L_{\tau^\delta}){\bf 1}_{\{\tau^\delta<\tau_{2n+1}\}}+\int_t^{\tau^\delta}\Mc_{s}^{t,\P'_\delta}dK^{\P'_\delta}_s\bigg],\; {\color{black}\tau_{2n}\leq t\leq\tau_{2n+1}},\; \P-a.s.$$
By definition of the $(\tau_n)_{n\geq 0}$, we deduce that
$$\delta\geq Y_{\tau_{2n}}-\Yc^{\P'_\delta}_{\tau_{2n}}\geq \eps\mathbb E^{\P'_\delta}_{\tau_{2n}}\Big[\Mc_{\tau^\delta}^{\tau_{2n},\P'_\delta}{\bf 1}_{\{\tau^\delta<\tau_{2n+1}\}}\Big]+\mathbb E^{\P'_\delta}_{\tau_{2n}}\bigg[\underset{\tau_{2n}\leq s\leq \tau^\delta}{\inf}\Mc_s^{\tau_{2n},\P'_\delta}\big(K^{\P'_\delta}_{\tau^\delta}-K^{\P'_\delta}_{\tau_{2n}}\big)\bigg].$$
We then estimate that
\begin{align*}
\P'_\delta[\tau^\delta< \tau_{2n+1}]&=\E^{\P'_\delta}\Big[\big(\Mc_{\tau^\delta}^{\tau_{2n},\P'_\delta}\big)^{-\frac12}\big(\Mc_{\tau^\delta}^{\tau_{2n},\P'_\delta}\big)^{\frac12}{\bf 1}_{\{\tau^\delta< \tau_{2n+1}\}}\Big]\\
&\leq \bigg(\E^{\P'_\delta}\Big[\big(\Mc_{\tau^\delta}^{\tau_{2n},\P'_\delta}\big)^{-1}\Big]\E^{\P'_\delta}\Big[\big(\Mc_{\tau^\delta}^{\tau_{2n},\P'_\delta}\big)^{\frac12}{\bf 1}_{\{\tau^\delta< \tau_{2n+1}\}}\Big]\bigg)^{\frac12}\\
&\leq C_{-1}^{\frac12}\sqrt{\frac\delta\eps}.
\end{align*}
Recall as well that by Step $(iii)$ of the proof of {\color{black}\cite[Theorem 3.1]{matoussi2013second}} that for some $\bar C>0$
$$\underset{\P'\in\Pc_H^\kappa(\tau_{2n}^+,\P)}{{\rm essup}^\P}\E^{\P'}_{\tau_{2n}}\big[\big(K^{\P'}_{\tau_{2n+1}}-K_{\tau_{2n}}^{\P'}\big)^2\big]\leq \bar C.$$
Therefore, we have
\begin{align*}
&\E^{\P'_\delta}_{\tau_{2n}}\Big[K^{\P'_\delta}_{\tau_{2n+1}}-K_{\tau_{2n}}^{\P'_\delta}\Big]\\
&\leq \E^{\P'_\delta}_{\tau_{2n}}\Big[K^{\P'_\delta}_{\tau^\delta}-K_{\tau_{2n}}^{\P'_\delta}\Big]+ \E^{\P'_\delta}_{\tau_{2n}}\Big[\big(K^{\P'_\delta}_{\tau_{2n+1}}-K_{\tau_{2n}}^{\P'_\delta}\big){\bf 1}_{\{\tau^\delta<\tau_{2n+1}\}}\Big] \\
&= \E^{\P'_\delta}_{\tau_{2n}}\bigg[\Big(\underset{\tau_{2n}\leq s\leq \tau^\delta}{\inf}\Mc_s^{t,\P}\Big(K^{\P'_\delta}_{\tau^\delta}-K_{\tau_{2n}}^{\P'_\delta}\Big)\Big)^\frac13\Big(K^{\P'_\delta}_{\tau^\delta}-K_{\tau_{2n}}^{\P'_\delta}\Big)^{\frac23}\Big(\underset{\tau_{2n}\leq s\leq \tau^\delta}{\inf}\Mc_s^{t,\P}\Big)^{-\frac13}\bigg]\\
&\hspace{0.9em}+\E^{\P'_\delta}_{\tau_{2n}}\Big[\big(K^{\P'_\delta}_{\tau_{2n+1}}-K_{\tau_{2n}}^{\P'_\delta}\big){\bf 1}_{\{\tau^\delta<\tau_{2n+1}\}}\Big]\\
&\leq  \bigg(\E^{\P'_\delta}_{\tau_{2n}}\bigg[\underset{\tau_{2n}\leq s\leq \tau^\delta}{\inf}\Mc_s^{t,\P}\Big(K^{\P'_\delta}_{\tau^\delta}-K_{\tau_{2n}}^{\P'_\delta}\Big)\bigg]\E^{\P'_\delta}_{\tau_{2n}}\bigg[\Big(K^{\P'_\delta}_{\tau^\delta}-K_{\tau_{2n}}^{\P'_\delta}\Big)^2\bigg]\E^{\P'_\delta}_{\tau_{2n}}\bigg[\Big(\underset{\tau_{2n}\leq s\leq \tau^\delta}{\inf}\Mc_s^{t,\P}\Big)^{-1}\bigg]\bigg)^{\frac13}\\
&\hspace{0.9em}+ \bigg(\E^{\P'_\delta}_{\tau_{2n}}\bigg[\Big(K^{\P'_\delta}_{\tau^\delta}-K_{\tau_{2n}}^{\P'_\delta}\Big)^2\bigg]\mathbb P'_\delta\big[\tau^\delta<\tau_{2n+1}\big]\bigg)^{\frac12}\\
&\leq \big(\bar CC_{-1}\big)^{\frac13}\delta^{\frac13}+\bar C^\frac12C_{-1}^\frac14\bigg(\frac{\delta}{\eps}\bigg)^\frac14.
\end{align*}
This implies immediately that 
$$\underset{\P'\in\Pc_H^\kappa(\tau_{2n}^+,\P)}{{\rm essinf}^\P}\E^{\P'}_{\tau_{2n}}\big[K^{\P'}_{\tau_{2n+1}}-K_{\tau_{2n}}^{\P'}\big]\leq \big(\bar CC_{-1}\big)^{\frac13}\delta^{\frac13}+\bar C^\frac12C_{-1}^\frac14\bigg(\frac{\delta}{\eps}\bigg)^\frac14,$$
which proves the second equality in \eqref{eq:stop} by letting $\delta$ go to $0$.

\medskip
Now, in order to prove that $\underset{\P\in\Pc_H^\kappa}{\inf}\E^\P\big[K^\P_{\tau_1}\big]=0$, notice that we just obtained
\[
\underset{\P'\in\Pc_H^\kappa(0^+,\P)}{{\rm essinf}^\P}\E^{\P'}_{0}\big[K^{\P'}_{\tau_{1}}\big]=0.
\]
Taking expectations, and using the fact that since the family of measure sis upward directed, the essential infimum is attained along some sequence $(\P_n)_{n\in\N}\subset \Pc_H^\kappa(0^+,\P)$, we deduce by the monotone convergence theorem under $\P$, and the fact that all the $(\P_n)_{n\in\N}$ coincide with $\P$ on $\Fc_{0+}$
\begin{align*}
0=\E^\P\bigg[\underset{\P'\in\Pc_H^\kappa(0^+,\P)}{{\rm essinf}^\P}\E^{\P'}_{0}\big[K^{\P'}_{\tau_{1}}\big]\bigg]=\E^\P\bigg[\lim_{n\to\infty}\downarrow\E^{\P_n}_{0}\big[K^{\P_n}_{\tau_{1}}\big]\bigg]=\lim_{n\to\infty}\downarrow\E^\P\Big[\E^{\P_n}_{0}\big[K^{\P_n}_{\tau_{1}}\big]\Big]& = \lim_{n\to\infty}\downarrow\E^{\P_n}\big[K^{\P_n}_{\tau_{1}}\big]\\
&\geq \underset{\P\in\Pc_H^\kappa}{\inf}\E^\P\big[K^\P_{\tau_1}\big],
\end{align*}
which proves the desired equality.

\vspace{0.5em}
Therefore, for any $n\geq 0$, we can find some $\P_1\in\Pc_H^\kappa$, and some $\P_{n+1}\in\Pc_H^\kappa(\tau_{2n}^+,\P_n)$ such that for some $\tilde C>0$
$$\E^{\P_{n+1}}_{\tau_{2n}}\big[K^{\P_{n+1}}_{\tau_{2n+1}}-K_{\tau_{2n}}^{\P_{n+1}}\big]\leq \frac{\eps}{2^n}.$$
By definition, we have $Y_t-L_t\leq 2\eps$ for $t\in[\tau_{2n-1},\tau_{2n}]$, so that
\begin{align*}
\mathbb E^{\P_n}\bigg[\int_0^{\tau_{2n}}\big(Y_{s^-}-L_{s^-}\big)dK_s^{\P_n}\bigg]&= \sum_{i=0}^{n-1}\mathbb E^{\P_n}\bigg[\int_{\tau_{2i+1}}^{\tau_{2(i+1)}}\big(Y_{s^-}-L_{s^-}\big)dK_s^{\P_n}+\int_{\tau_{2i}}^{\tau_{2i+1}}\big(Y_{s^-}-L_{s^-}\big)dK_s^{\P_n}\bigg]\\
&=\sum_{i=0}^{n-1}\mathbb E^{\P_{i+1}}\bigg[\int_{\tau_{2i+1}}^{\tau_{2(i+1)}}\big(Y_{s^-}-L_{s^-}\big)dK_s^{\P_{i+1}}+\int_{\tau_{2i}}^{\tau_{2i+1}}\big(Y_{s^-}-L_{s^-}\big)dK_s^{\P_{i+1}}\bigg]\\
&\leq \sum_{i=0}^{n-1}2\eps\mathbb E^{\P_n}\Big[K_{\tau_{2(i+1)}}^{\P_{i+1}}-K_{\tau_{2i+1}}^{\P_{i+1}}\Big]+\frac{\eps}{2^i}\mathbb E^{\P_n}\bigg[\underset{\tau_{2i}\leq s\leq \tau_{2i+1}}{\sup}\big(Y_s-L_s\big)\bigg]\\
&\leq \tilde C\eps,
\end{align*}
where we used the {\it a priori} estimates satisfied by the solution of the 2{\rm RBSDE}, see \cite[Theorem 3.3]{matoussi2013second} and the definition of $\P_n$.

\vspace{0.5em}
Next, notice that $Y-L$ is right--continuous, and therefore uniformly continuous from the right (see \cite[Section 2.8]{applebaum2004levy}). Besides, by definition, we have for any $n\geq 0$, that on $\{\tau_{n+1}<T\}$
$$\abs{(Y_{\tau_{n+1}^-}-L_{\tau_{n+1}^-})-(Y_{\tau_{n}^-}-L_{\tau_{n}^-})}\geq \eps.$$
Therefore the $\tau_n$ cannot accumulate and for $n$ large enough we necessarily have $\tau_n=T$. We now assume that the $n$ we have chosen satisfies this property.

\vspace{0.5em}
Finally, fix some $m\leq n$. We have the following estimate for any $\P\in\Pc_H^\kappa$
\begin{align*}
\P[\tau_{2n}<T]&\leq \P\bigg[\bigcap_{i=0}^{n-1}\Big\{\big|Y_{\tau_{2(i+1)}^-}-Y_{\tau_{2i+1}^-}\big|+\big|L_{\tau_{2(i+1)}^-}-L_{\tau_{2i+1}^-}\big|\geq \eps\Big\}\bigg]\\
&\leq \P\bigg[\bigcap_{i=0}^{n-1}\bigg(\Big\{\big|Y_{\tau_{2(i+1)}^-}-Y_{\tau_{2i+1}^-}\big|\geq\frac\eps2\Big\}\bigcup\Big\{\big|L_{\tau_{2(i+1)}^-}-L_{\tau_{2i+1}^-}\big|\geq \frac\eps2\Big\}\bigg)\bigg]\\
&\leq \P\bigg[\bigg\{\Sum_{i=0}^{n-1}\big|Y_{\tau_{2(i+1)}^-}-Y_{\tau_{2i+1}^-}\big|^2\geq \frac{m\eps^2}{4}\bigg\}\bigcup\bigg\{\Sum_{i=0}^{n-1}{\bf 1}_{\big\{\big|L_{\tau_{2(i+1)}^-}-L_{\tau_{2i+1}^-}|\geq \eps/2\big\}}\geq (n-m)\bigg\}\bigg]\\
&\leq \frac{4}{m\eps^2}\E^\P\bigg[\Sum_{i=0}^{n-1}\big|Y_{\tau_{2(i+1)}^-}-Y_{\tau_{2i+1}^-}\big|^2\bigg]+\P\bigg[\sum_{i=0}^{n-1}{\bf 1}_{\big\{|L_{t_{i+1}^{n-}}-L_{t_i^{n-}}|\geq\eps/2\big\}}\geq n-m\bigg].
\end{align*}
Now notice that we have for some constant $C$ which may change value from line to line, by definition and using Doob's inequality as well as the elementary inequality $\sum_i a_i^2\leq(\sum_i |a_i|)^2$ and the estimates of \cite[Theorem 3.3]{matoussi2013second}
\begin{align*}
\E^\P\bigg[\Sum_{i=0}^{n-1}\big|Y_{\tau_{i+1}^-}-Y_{\tau_i^-}\big|^2\bigg]\leq C\mathbb E^\P\bigg[\int_0^T\big|\widehat F_s(Y_s,Z_s)\big|^2ds+\int_0^T\big|\widehat a_s^{1/2}Z_s\big|^2ds+\big(K_T^\P\big)^2\bigg]\leq C.
\end{align*}
Consequently, using Assumption \ref{assum}
\begin{align*}
\underset{\mathbb P\in\Pc^\kappa_H}{{\rm inf}}\; \mathbb E^{\P}\bigg[\int_0^T\big(Y_{s^-}-L_{s^-}\big)dK_s^{\P}\bigg]&\leq  \mathbb E^{\P_n}\bigg[\int_0^T\big(Y_{s^-}-L_{s^-}\big)dK_s^{\P_n}\bigg]\\
&\leq \tilde C\eps+\bigg(\mathbb E^{\P_n}\bigg[\underset{0\leq s\leq T}{\sup}\big(Y_s-L_s\big)^2\bigg]\mathbb E^{\P_n}\Big[\big(K_T^{\P_n}\big)^2\Big]\bigg)^\frac12\P^n[\tau_{2n}<T]\\
&\leq \tilde C\eps +\frac{4C}{m\eps^2}+\P\bigg[\sum_{i=0}^{n-1}{\bf 1}_{\big\{|L_{t_{i+1}^{n-}}-L_{t_i^{n-}}|\geq\eps/2\big\}}\geq n-m\bigg].
\end{align*}
It thus suffices to let $n$ go to $+\infty$ first, then $m$ to $+\infty$ and finally $\eps$ to $0$.
\end{proof}

\subsubsection{Comparison with the literature}
In the recent months, two independent studies of the so-called reflected $G-$BSDEs have appeared, the first by Li and Peng \cite{li2017reflected}, and the second in the PhD thesis of Soumana Hima \cite{soumana2017equations}. Both these papers obtain wellposedness, in the $G-$framework of Peng of solutions to reflected $G-$BSDEs with a lower obstacle. Unlike our first paper \cite{matoussi2013second}, they ensure uniqueness by using the Skorokhod minimality condition \eqref{eq:sko}. However, as shown by the result of the previous section, under Assumption \ref{assum}, both minimality conditions actually lead to the exact same solution. Let us now detail a bit more the other differences between the two different approaches.

\begin{itemize}
\item[$(i)$] First of all, concerning the assumptions made, the main difference is on the obstacle. In \cite[Theorems 5.1 and 5.2]{li2017reflected}, in addition to our own assumptions, it is assumed to either be bounded from above or that it is a semimartingale under every measure considered (see their Assumptions $(H4)$ and $(H4')$). Similarly, \cite{soumana2017equations} requires the obstacle to be a semimartingale (see the equation just after $(5.4)$ in \cite{soumana2017equations}). In our framework, if one is satisfied with Definition \ref{def1}, then we only require classical square integrability on $L$. If one also wants to recover the Skorokhod condition, then we need more in the form of Assumption \ref{assum}. In any case, this does not imply that $L$ has to be a semimartingale nor bounded from above, and as shown in Lemma \ref{lemma}, it would be enough for $L$ to have finite $p-$variation for some $p\geq 1$, which is obviously satisfied if $L$ is a semimartingale, making our assumption weaker in general.

\item[$(ii)$] Concerning the method of proof, both \cite{li2017reflected} and \cite{soumana2017equations} use the classical penalisation method introduced by \cite{el1997reflected} to prove existence, while uniqueness is obtained through {\it a priori} estimates. Our proof is more constructive and in the spirit of the original paper \cite{soner2012wellposedness}. We expect that the penalisation approach should be applicable in our setting as well, but we leave this interesting question to future research. 

\item[$(iii)$] Maybe more important than the above point, one has to keep in mind that the very essence of the $G-$BSDE theory requires that the data of the equation, meaning here the generator $\widehat F$, the terminal condition $\xi$ and the obstacle $L$, have to have some degree of regularity with respect to the $\omega$ variable. More precisely, they have to be quasi-continuous in $\omega$, which loosely speaking means that they must be uniformly continuous (for the uniform convergence topology) outside a "small" set (see the references for more details). This is inherent to the construction itself, as soon as the set $\Pc_H^\kappa$ is non--dominated, and cannot be avoided with this approach. Granted, it is also the case in our paper \cite{matoussi2013second}. However, since then, many progresses have been achieved in the {\rm 2BSDE} theory, and the recent paper \cite{possamai2018stochastic} has proved that the (non--reflected) {\rm 2BSDE} theory worked perfectly without {\it any} regularity assumption. Furthermore, a general {\it modus operandi} is given in \cite[Proposition 2.1 and Remark 4.2]{possamai2018stochastic} to extend those results to many type of {\rm {\rm 2BSDE}s}, including the reflected ones. This program has actually been carried out in the recent PhD thesis Noubiagain \cite{noubiagain2017equations}  (see also \cite{denis2017generalized2} and \cite{matoussi2018generalized}). Combined with the results and discussions of the present note, the {\rm 2RBSDE}s can therefore be defined in a much more general framework than the reflected $G-$BSDEs.

\end{itemize}

\section{A super--hedging duality for American options in uncertain, incomplete and nonlinear markets}\label{sec:american}

This short section is devoted to obtain some clarifications concerning the link made in \cite{matoussi2013second} between solutions to 2RBSDEs and super--hedging prices for American options under volatility uncertainty. The recent years have seen a flourishing of papers treating the above super--replication problem of American options in {discrete time} financial markets under uncertainty, allowing or not for static trading of European options, see among others Dolinsky \cite{dolinsky2017numerical}, Neuberger \cite{neuberger2007bounds}, Hobson and Neuberger \cite{hobson2016more,hobson2016value}, Bayraktar et al. \cite{bayraktar2017hedging}, Bayraktar and Zhou \cite{bayraktar2017super}, or Deng and Tan \cite{deng2016duality}, and Aksamit, Deng, Ob\l{}\'oj and Tan \cite{aksamit2016robust}. In a continuous--time setting, non--linear markets were considered by Dumitrescu, Quenez and Sulem \cite {dumitrescu2017game,dumitrescu2018american}, with some level of ambiguity, in the sense that both the non--linear driver of the wealth process and the default intensity could be not perfectly known. This however corresponds to families of probabilities which are absolutely continuous with respect to each other, and thus does not require to consider second--order BSDEs. As far as we know, beyond the results given in \cite{matoussi2013second}, there are no other results allowing to tackle volatility uncertainty in continuous--time (see however the recent contribution \cite{dumitrescu2017bsdes} for partial hedging issues). 

\medskip
Given that this is a corrigendum, we do not wish to go into too many details, and simply want to point out that \cite{matoussi2013second} only provided an upper bound for the super--hedging price of an American options as the initial value of a 2RBSDE, and that using techniques similar to the ones in \cite{dumitrescu2017game} for instance, it can readily be checked that this is actually the super--hedging price itself.

\section{General reflections}

\subsection{Uniqueness}
Let us now consider our second paper \cite{matoussi2014second}. First of all, the definition of a solution should be replaced by the following.
\begin{Definition}\label{death2}
We say $(Y,Z)\in \mathbb D^{2,\kappa}_H\times\mathbb H^{2,\kappa}_H$ is a solution to a {\rm 2DRBSDE} if

\begin{itemize}
\item[$\bullet$] $Y_T=\xi$, $\mathcal P_H^\kappa-q.s$.

\vspace{0.35em}
\item[$\bullet$] $\forall \mathbb P \in \mathcal P_H^\kappa$, the process $V^{\mathbb P}$ defined below has paths of bounded variation $\mathbb P-a.s.$
\begin{equation}
V_t^{\mathbb P}:=Y_0-Y_t - \int_0^t\widehat{F}_s(Y_s,Z_s)ds+\int_0^tZ_sdB_s, \text{ } 0\leq t\leq T, \text{  } \mathbb P-a.s.,
\label{2dRBSDE.kref}
\end{equation}
and admits the following decomposition
\begin{equation}\label{eq:decomp}
V_t^{\mathbb P}=K_t^{\mathbb P}-\mathcal K_t^{\mathbb P,+},\;t\in[0,T],\; \P-a.s.,
\end{equation}
where the two processes $K^\P$ and $\mathcal K^{\P,+}$ are non-decreasing, and where $\mathcal K^{\P,+}$ satisfies the following Skorokhod condition
\begin{equation}\label{eq:sko2}
\int_0^T\big(S_{s^-}-Y_{s^-}\big)d\mathcal K^{\P,+}_s=0,\; \P-a.s.
\end{equation}
\item[$\bullet$] We have the following minimality condition for $0\leq t\leq T$
\begin{equation}\label{eq:new2}
\underset{\mathbb P'\in\Pc^\kappa_H(t^+,\P)}{{\rm essinf}^\P}\; \mathbb E_t^{\P'}\left[\int_t^TM^{t,\P'}_sd\big(V^{\P'}_s+k_s^{\P',+}-k^{\P',-}_s\big)\right]=0, \text{  } \mathbb P-a.s., \; 0\leq t\leq T, \text{ } \forall \mathbb P \in \mathcal P_H^\kappa,
\end{equation}
where $M^{t,\P}$ is defined as in \eqref{eqM} but using the solution $(y^\P,z^\P)$ of the doubly reflected BSDE under $\P$.
\item[$\bullet$] $L_t\leq Y_t\leq S_t$, $\mathcal P_H^\kappa-q.s.$
\end{itemize}
\end{Definition}

There are two main differences with the earlier definition in our paper \cite{matoussi2014second}. The first one is obviously the new minimality condition \eqref{eq:new2}, which is simply the version with two obstacles of \eqref{eq:new}. The second main difference is the decomposition \eqref{eq:decomp} of the bounded variation process $V^\P$. It is not really new, {\it per se}, as it was already implicit in the existence proof we provided in \cite{matoussi2014second}, see in particular the lignes between the statements of Lemma 4.3 and Proposition 4.4. In particular, it does not require any additional argument in the existence proof. 

\vspace{0.5em}
Under this new definition, the proof of uniqueness of a solution follows exactly the same lignes as in the lower obstacle case described above, it suffices to use the new minimality condition \eqref{eq:new2}, which is equivalent to the representation formula of the solution to the 2D{\rm RBSDE} as an essential supremum of solutions of the associated D{\rm RBSDE}s.

\subsection{{\it A priori} estimates}
The main change in \cite{matoussi2014second} with the introduction of the new minimality condition \eqref{eq:new2} above concerns the {\it a priori} estimates for 2D{\rm RBSDE}s. Let us start with Proposition 3.5 in \cite{matoussi2014second}, which has to be corrected as follows. Notice that the references (2.5) and (2,6) are the ones from \cite{matoussi2014second}, and not the present paper.

\begin{Proposition}
Let Assumption $2.3$ hold. Assume $\xi\in\mathbb L^{2,\kappa}_H$ and $(Y,Z)\in\mathbb D^{2,\kappa}_H\times\mathbb H^{2,\kappa}_H$ is a solution to the {\rm 2DRBSDE} $(2.5)$. Let $\{(y^\P,z^\P,k^{\P,+},k^{\P,-})\}_{\P\in\Pc^\kappa_H}$ be the solutions of the corresponding {\rm DRBSDE}s $(2.6)$. Then we have the following results for all $t\in[0,T]$ and for all $\P\in\Pc^\kappa_H$
\begin{itemize}[leftmargin=*]
\item[$(i)$] $\displaystyle V_t^{\P,+}:=\int_0^t{\bf 1}_{Y_{s^-}=L_{s^-}}dV_s^\P=\int_0^t{\bf 1}_{Y_{s^-}=L_{s^-}}dk_s^{\P,-}$, $\P-a.s.,$ and is therefore a non--decreasing process.
\item[$(ii)$] $\displaystyle V_t^{\P,-}:=\int_0^t{\bf 1}_{y^\P_{s^-}=S_{s^-}}dV_s^\P=-\int_0^t{\bf 1}_{y^\P_{s^-}=S_{s^-}}dk_s^{\P,+}$, $\P-a.s.,$ and is therefore a non--increasing process.
\end{itemize}
\end{Proposition}
The proof of $(ii)$ above is given in \cite{matoussi2014second} and does not use the minimality condition and is thus correct. $(i)$ can be proved similarly. The issue now is that we no longer have a nice Jordan decomposition of $V^\P$, which changes a lot how we can prove and obtain {\it a priori} estimates for the solution. 

\vspace{0.5em}
Actually, the main point here is to rely on the decomposition \eqref{eq:decomp}, which is almost a Jordan decomposition. In the proof of Theorem 3.7 in \cite{matoussi2014second}, the proof of the estimates for $Y$, $y^\P$, $z^\P$, $k^{\P,+}$ and $k^{\P,-}$ does not change and is still correct. In the estimate for $Z$, corresponding to the calculations in $(3.15)$ in \cite{matoussi2014second}, one has to use the decomposition \eqref{eq:decomp} for $V^\P$, and the fact that we know that for some constant $C$ independent of $\P$
$$\mathbb E^\P\big[|K_T^\P|^2+|\Kc_T^\P|^2\big]\leq C.$$
Indeed, this is a consequence of \cite[Lemma A.11]{matoussi2014second} and the fact that the unique solution to the 2D{\rm RBSDE} is constructed through the Doob--Meyer decomposition of a doubly reflected $g-$supermartingale. The rest of the proof is then the same, still using the decomposition \eqref{eq:decomp}. Thus Theorem 3.7 in \cite{matoussi2014second} should be replaced by

\begin{Theorem}\label{destimatesref}
Let Assumptions 2.3, 2.5 and 2.8 hold. Assume $\xi\in\mathbb L^{2,\kappa}_H$ and $(Y,Z)\in \mathbb D^{2,\kappa}_H\times\mathbb H^{2,\kappa}_H$ is a solution to the 2D{\rm RBSDE} $(2.5)$. Let $\left\{(y^\mathbb P,z^\mathbb P,k^{\mathbb P,+},k^{\mathbb P,-})\right\}_{\mathbb P\in\mathcal P^\kappa_H}$ be the solutions of the corresponding {\rm DRBSDE}s $(2.6)$. Then, there exists a constant $C_\kappa$ depending only on $\kappa$, $T$ and the Lipschitz constant of $\widehat F$ such that
\begin{align*}
&\No{Y}^2_{\mathbb D^{2,\kappa}_H}+\No{Z}^2_{\mathbb H^{2,\kappa}_H}+\underset{\mathbb P\in \mathcal P^\kappa_H}{\sup}\left\{\big\|y^\mathbb P\big\|^2_{\mathbb D^{2}(\mathbb P)}+\big\|z^\mathbb P\big\|^2_{\mathbb H^{2}(\mathbb P)}\right\}\\
&\hspace{0.9em}+\displaystyle\underset{\mathbb P\in \mathcal P^\kappa_H}{\sup}\mathbb E^\mathbb P\left[{\rm Var}_{0,T}\big(V^\P\big)^2+\big(K_T^\P\big)^2+\big(\Kc_T^{\P,+}\big)^2+\big(k_T^{\mathbb P,+}\big)^2+\big(k_T^{\mathbb P,-}\big)^2\right]\leq C_\kappa\left(\No{\xi}^2_{\mathbb L^{2,\kappa}_H}+\phi^{2,\kappa}_H+\psi^{2,\kappa}_H+\varphi^{2,\kappa}_H+\zeta^{2,\kappa}_H\right).
\end{align*}
\end{Theorem}

Next, concerning the estimates for the difference between two solutions, the proof of Theorem 3.8 in \cite{matoussi2014second} also has to be modified. More precisely, the three lignes after (3.19) should be erased. Then the proof of the estimate for $\delta Z$ is still correct. However, we only have control over the difference between $V^{\P,1}$ and $V^{\P,2}$, not individually for $K^{\P,1}$ and $K^{\P,2}$ on the one hand, and $\Kc^{\P,1}$ and $\Kc^{\P,2}$ on the other hand. Theorem 3.8 of \cite{matoussi2014second} should therefore be replaced by
\begin{Theorem}\label{estimates2d}
Let Assumptions 2.3, 2.5 and 2.8 hold. For $i=1,2$, let $(Y^i,Z^i)$ be the solutions to the 2D{\rm RBSDE} $(2.5)$ with terminal condition $\xi^i$, upper obstacle $S$ and lower obstacle $L$. Then, there exists a constant $C_\kappa$ depending only on $\kappa$, $T$ and the Lipschitz constant of $F$ such that
\begin{align*}
&\No{Y^1-Y^2}_{\mathbb D^{2,\kappa}_H}\leq C\No{\xi^1-\xi^2}_{\mathbb L^{2,\kappa}_H}\\
&\No{Z^1-Z^2}^2_{\mathbb H^{2,\kappa}_H}+\underset{\mathbb P\in \mathcal P^\kappa_H}{\sup}\mathbb E^\mathbb P\left[\underset{0\leq t\leq T}{\sup}\abs{V_t^{\mathbb P,1}-V_t^{\mathbb P,2}}^2\right]\\
&\leq C\No{\xi^1-\xi^2}_{\mathbb L^{2,\kappa}_H}\left(\No{\xi^1}_{\mathbb L^{2,\kappa}_H}+\No{\xi^1}_{\mathbb L^{2,\kappa}_H}+(\phi^{2,\kappa}_H)^{1/2}+(\psi^{2,\kappa}_H)^{1/2}+(\varphi^{2,\kappa}_H)^{1/2}+(\zeta^{2,\kappa}_H)^{1/2}\right).
\end{align*}

\end{Theorem}
Notice also that Remark $3.9$ in \cite{matoussi2014second} no longer holds. Similarly, Remark 3.12 should be deleted. As a consequence, in Proposition 3.10 in \cite{matoussi2014second}, the constant $\gamma$ should always be taken as equal to $0$. Finally, direct computations using the decomposition \eqref{eq:decomp} prove that Proposition 3.14 in \cite{matoussi2014second} should be replaced by
\begin{Proposition}
Let Assumptions 2.3, 2.5, 2.8 and 3.13 hold. Let $(Y,Z)$ be the solution to the $2$D{\rm RBSDE}, then for all $\mathbb P\in\mathcal P^\kappa_H$
\begin{align}
&Z_t=P_t, \text{ }dt\times\mathbb P-a.s.\text{ on the set }\left\{Y_{t^-}=S_{t^-}\right\},
\end{align}
and there exists a progressively measurable process $(\alpha_t^\mathbb P)_{0\leq t\leq T}$ such that $0\leq \alpha\leq 1$ and
$$d\Kc_t^{\mathbb P,+}=\alpha_t^\mathbb P{\bf 1}_{Y_{t^-}=S_{t^-}}\left(\left[\widehat F_t(S_t,P_t)+U_t\right]^+dt+dC_t^++dK^\P_t\right).$$
\end{Proposition}
  
\subsection{Existence}
Because we no longer control the total variation of $V^\P$ in Theorem \ref{estimates2d}, the proof of existence we gave in \cite{matoussi2014second} only holds for $\xi\in{\rm UC}_b(\Omega)$. However, this is not an issue at all, since the only reason we had to restrict to uniformly continuous terminal condition was to obtain the measurability result in \cite[Lemma 4.1]{matoussi2014second} and the dynamic programming principle of \cite{matoussi2014second}. Using the results of \cite{possamai2018stochastic}, in particular Proposition 2.1, these two results were obtained in \cite{noubiagain2017equations} (see also \cite{matoussi2018generalized}) for doubly reflected BSDEs, and allow to extend the construction carried out in \cite{matoussi2014second} to any $\xi\in\mathbb L^{2,\kappa}_H$.

\vspace{0.5em}
Finally, notice that similar arguments as in the lower reflected case should in principle allow to prove that wellposedness can be recovered for 2D{\rm RBSDE}s when the minimality condition \eqref{eq:new2} is replaced by asking that the process $K^\P$ in the decomposition \eqref{eq:decomp} satisfies some sort of Skorokhod condition similar to \eqref{eq:sko}, and provided that an conditions similar to Assumption \ref{assum} hold. In such a situation, both processes $K^\P$ and $\Kc^{\P,+}$ would then satisfy some Skorokhod type conditions. Such a program has been carried out recently for $G$--RBSDEs with two obstacles in \cite{li2019backward}, and \cite{li2020reflected}.

\subsection{Game options}
In Section 5.1 of \cite{matoussi2014second}, we introduced game options and claimed that the second order doubly reflected BSDEs (2D{\rm RBSDE}s for short) allow us to obtain super-- and sub--hedging prices for game options in financial markets with
volatility uncertainty. Actually,  it is proved that the amount $Y_t$, where $Y$ is the solution of the 2D{\rm RBSDE} in Definition 3.1 of \cite{matoussi2014second}, allows the seller of the game option to build a super--hedging strategy under any probability measure $\P'$.   
We emphasise however that we are not able to guarantee that this amount is optimal in the sense that it is the lowest value for which we can find a super--hedging strategy, though, as explained in the case of American options above in Section \ref{sec:american}, we strongly expect this result to hold. 

\medskip
For related problems in nonlinear market with default, but without volatility uncertainty, we refer the reader to the recent paper \cite{dumitrescu2017game} and the references therein. Note that \cite{dumitrescu2017game} studies an associated (non-linear) robust Dynkin game problem, in particular, in the case when there is default intensity ambiguity on the model. 


\vspace{0.5em}
Moreover, we have claimed  in  Section 5.1 of \cite{matoussi2014second}, that the whole interval of prices, given by $[\widetilde Y_t , Y_t ]$ with $\widetilde Y_t :=\underset{\mathbb P'\in\Pc^\kappa_H(t^+,\P)}{{\rm essinf}^\P}\; y^{\P'}_t$ (Page 2309 in \cite{matoussi2014second}), can be formally considered as arbitrage free. This also requires a proper justification. Actually, we may define the super--hedging price for a game option as in Section 6.1 of \cite{dolinsky2017numerical}. Using this definition, the link between the super--hedging price and the solution of a {\rm 2DRBSDE} will be considered in the forthcoming working paper \cite{matoussi2018generalized}.

 \bibliographystyle{plain}
 \small
\bibliography{bibliographyDylan}

 \end{document}